\documentclass{aptpub}
\usepackage[numbers]{natbib}
\authornames{B. BASRAK, N. MILIN\v{C}EVI\'C, P. \v{Z}UGEC} % insert the authors here for use in running head
\shorttitle{On extremes of random clusters and marked renewal cluster processes} % insert short title here for use in running head

\usepackage{bbm, bm, color}

\newcommand{\dsum}{\displaystyle\sum}
\newcommand{\dint}{\displaystyle\int}

\newcommand{\cid}{\stackrel{d}{\longrightarrow}}
\newcommand{\cip}{\stackrel{P}{\longrightarrow}}
\newcommand{\cas}{\stackrel{as}{\longrightarrow}}

\newcommand{\toi}{\to\infty}
\newcommand{\eind}{\stackrel{d}{=}}
 
\newcommand{\Exp}{\mathbb{E}}

\newcommand{\vep}{\varepsilon}
\newcommand{\taui}{\Gamma_i}
\newcommand{\taun}{\Gamma_n}

\newcommand{\restr}[1]{|_{#1}}

\newtheorem{innercustomgeneric}{\customgenericname}
\providecommand{\customgenericname}{}
\newcommand{\newcustomtheorem}[2]{%
  \newenvironment{#1}[1]
  {%
   \renewcommand\customgenericname{#2}%
   \renewcommand\theinnercustomgeneric{##1}%
   \innercustomgeneric
  }
  {\endinnercustomgeneric}
}
\numberwithin{equation}{section}

\newtheorem{theo}{Theorem}[section]
\newtheorem{lemm}[theo]{Lemma}
\newtheorem{coro}[theo]{Corollary}
\newtheorem{examp}[theo]{Example}
\newtheorem{propo}[theo]{Proposition}
\newtheorem{rema}[theo]{Remark}

\newcustomtheorem{customexample}{Example}

\newcommand{\BB}{\mathcal{B}}

\newcommand{\NN}{\mathbb{N}}

\newcommand{\PP}{\mathbb{P}}
\newcommand{\R}{\mathbb{R}}

\newcommand{\ZZ}{\mathbb{Z}}
\newcommand{\1}{\mathbb{I}}
\newcommand{\RR}{\mathbb{R}}

\newcommand{\EE}{\mathbb{E}}

\DeclareMathOperator{\PRM}{PRM}
\DeclareMathOperator{\MDA}{MDA}
\usepackage[dvipsnames]{xcolor}

\begin{document}

\title{On extremes of random clusters and marked renewal cluster processes} % insert title - use \\ if it requires more than one line.

\authorone[University of Zagreb]{Bojan Basrak} % Affiliation is just the name of your university or institution
\addressone{Department of Mathematics, University of Zagreb, Bijeni\v{c}ka 30, Zagreb, Croatia} % Your postal address goes here.
%\emailone{bbasrak@math.hr}
\authortwo[University of Zagreb]{Nikolina Milin\v{c}evi\'c} % Affiliation is just the name of your university or institution
\addresstwo{Department of Mathematics, University of Zagreb, Bijeni\v{c}ka 30, Zagreb, Croatia} % Your postal address goes here.
%\emailtwo{olivier.wintenberger@upmc.fr}
\authorthree[University of Zagreb]{Petra \v{Z}ugec} % Affiliation is just the name of your university or institution
\addressthree{Faculty of Organization and Informatics, Pavlinska 2, Vara\v{z}din, University of Zagreb, Croatia} % Your postal address goes here.
%\emailthree{petra.zugec@foi.hr}

\begin{abstract}

The article describes the limiting distribution of the extremes of observations that arrive in clusters. We start by studying the tail behaviour of an individual cluster and then we apply the developed theory to determine the limiting distribution of $\max\{X_j: j=0,\ldots, K(t)\}$, where  $K(t)$ is the number of i.i.d. observations $(X_j)$ arriving up to the time $t$ according to a general marked renewal cluster process. The results are illustrated in the context of some commonly used Poisson cluster models such as the marked Hawkes process.

\end{abstract}

 % insert keywords separated by a semicolon

 % insert the primary Maths Subject Classification number in the first bracket
         % and the secondary ams number(s) in the second bracket
         % e.g. \ams{60E20}{49G03;49F10}

\section{Introduction}\label{sec:intro} % Initial capital letter, then lower case. No full stop.

In many real life situations one encounters observations  which tend to cluster when collected over time. This behaviour is commonly seen in various applied fields including, for instance, 
non-life insurance, 
climatology 
and hydrology (e.g. \cite{mikosch}, \cite{vogel et al} and \cite{towe et al}).
This article aims to describe the limiting distribution for the extremes of such  observations over increasing time intervals.

In Section~\ref{sec:RM} we study a simpler question concerning the tail behaviour of the maximum in one random cluster of observations. More precisely, consider
$$ H = \bigvee_ {j=1}^K X_j\,,$$
where we assume that  i.i.d. sequence $(X_j)$  belongs to the maximum domain of attraction of some extreme value distribution $G$, or $\MDA(G)$ for short, and $K$ is a positive random integer, possibly dependent on the observations themselves. For  an introduction to $\MDA$'s and extreme value theory in general we refer to \cite{res87},  \cite{embrechtsetal}   or \cite{dehaan}.  In the case of non-random $K$, $H$ belongs to the same $\MDA$ as $X_1$ by the standard extreme value theory. The case of $K$ independent of the sequence $(X_j)$ has been subject of several studies including \cite{JesMik} and \cite{tillier-wint}, see also \cite{fayetal} where the tail behaviour of the randomly indexed sums are studied in a similar setting. The same problem in a multidimensional setting is recently considered in \cite{hashorva}. In the sequel, we allow for instance for $K$ to be a stopping time with respect to  the sequence $(X_j)$ and we show that  $H$ remains in the same $\MDA$ as the observations as long as $K$ has a finite mean. This is the content of our main theorem in Section~\ref{sec:RM}. For this result we provide an original and relatively simple proof based on \cite{drago}.

In Section~\ref{sec:GMEM} we consider observations $(X_j)$ which are i.i.d. but arrive in possibly overlapping groups at times $\tau_1,\tau_2,\ldots$. We show how one can determine the asymptotic distribution of $ M(t) = \sup \left\lbrace X_k: \tau_{k} \leq t\right\rbrace$ under certain mild conditions on the clustering among the observations. Due to the results in Section~\ref{sec:RM}, it turns out that the effects of clustering often remain relatively small in the limit, cf.  Corollary \ref{cor:41}. 
Processes of the form: $M(t) = \bigvee_ {j=0}^{K(t)} X_j$, where $K(t)$ is a stochastic process possibly dependent on the observations $X_j$, have	received considerable attention over the years. For some of the earliest contributions see \cite{berman} and \cite{barndorff}.

More recently, \cite{meerschaert} and \cite{pancheva}
studied the convergence of process $(M(t))$ towards an appropriate extremal process. 
For the study of all upper order statistics up to time $K(t)$ see \cite{drago} and for the more general weak convergence of extremal processes with a random sample size see \cite{silvestrov}.

Section~\ref{sec:APP} is dedicated to application of our main results to some  frequently used stochastic models of clustering. In particular, we study variants of Neyman-Scott, Bartlett-Lewis and randomly marked Hawkes process. For each of the three clustering mechanisms we find sufficient conditions which imply that $M(t)$ properly centred and normalized,  roughly speaking, stays in the $\MDA(G)$.

	Throughout, let $\mathbb{S}$ denote a general Polish space and $\BB(\mathbb{S})$ a Borel $\sigma$-algebra on $\mathbb{S}$.  The space of boundedly finite point measures on $ \mathbb{S}$ is denoted by $M_p ( \mathbb{S})$.
For this purpose $\mathbb{S}$ is endowed with a family of the so--called {\em bounded} sets, see \cite{baplSPL}. 
We use the standard vague topology on the space $M_p(\mathbb{S})$ (see~\cite{res87} or \cite{kallenberg2017}).  Recall that  $m_n \stackrel{v}{\longrightarrow} m$ in $M_p( \mathbb{S})$ simply means that  $\int f dm_n \longrightarrow \int f  dm$ for any bounded continuous function $f: \mathbb{S} \to \RR$
whose support is {\em bounded} in the space $\mathbb{S}$.

The Lebesgue measure will be denoted by $\lambda$, whereas the Poisson random measure with mean measure $\eta$ will be denoted by $\PRM(\eta)$.
To simplify the notation, for a generic member of an identically distributed sequence or an array, say $(X_j)$, $(A_{i,j})$, throughout  we write $X,\, A\,$ etc. The set of natural numbers will be denoted by $\NN = \{1,2,\dots\}$. The set of non-negative integers we denote by $\ZZ_+$.

\section{Random maxima}\label{sec:RM} 

Let $(X_j)_{j\in \NN}$ be  an i.i.d. sequence with distribution belonging to $\MDA(G)$ where $G$ is one of the three extreme value distributions, and let $K$ denote a random non-negative integer. We are interested in the tail behaviour of
$$ H = \bigvee_ {j=1}^K X_j\,.$$

In the sequel we allow that $K$ depends on the values of sequence $(X_j)_{j\in \NN}$ together with some additional sources of randomness.
Assume that $((W_j, X_j))_{j\in \NN}$ is a sequence of i.i.d. random elements in $\mathbb{S}\times \RR$. For the filtration $(\mathcal{F}_n)_{n\in \NN} = (\sigma\{(W_j, X_j): j\leq n \})_{n\in \NN}$ we assume that $K$ is a stopping time with respect to $(\mathcal{F}_n)_{n\in \NN}$. In this case already, $H$ can be a rather complicated distribution, as one can see from the following.

\begin{examp}\label{example of stopping times} 
\begin{itemize}
    \item[(a)] Assume $(W_j)_{j\in \NN}$ is independent of $(X_j)_{j\in \NN}$ and integer valued. When $K=W_1$, $H$ has been studied already in the references mentioned in the Introduction.
    \item[(b)] Assume $((W_j, X_j))_{j\in \NN}$ is i.i.d. as before (note that some mutual dependence between $W_j$ and $X_j$ is allowed) and $\PP(X>W)>0$. Let $K=\inf\{k\in \NN: X_k>W_k\}$. Clearly $K$ has geometric distribution and we will show that this implies that $H$ is in the same $\MDA$ as $X$. 
    \item[(c)] Assume $(W_j)_{j\in \NN}$ and $(X_j)_{j\in \NN}$ are two independent i.i.d. sequences. Let $K=\inf \{k\in \NN: X_k > W_1\}$. Clearly $H=X_{K}>W_1$. Therefore, $H$ has a tail at least as heavy as $W$.
\end{itemize}
\end{examp}

Recall (see Chapter $1$ in \cite{res87} by Resnick) that the assumption that $X$ belongs to $\MDA(G)$ is equivalent to the existence of a sequence of positive real numbers $(a_n)_{n\in \NN}$ and a sequence of real numbers $(b_n)_{n\in \NN}$ such that for every $x\in \EE = \{y\in \RR: G(y)>0\}$ 
\begin{align}\label{mda(g)}
    n \cdot \PP(X> a_n \cdot x + b_n)\to -\log G(x) \quad \text{ as } n\to \infty,
\end{align}
and it is further equivalent to 
\begin{equation*}
    \PP\left(\frac{\bigvee_{i=1}^n X_i - b_n}{a_n}\leq x\right)\to G(x) \quad \text{ as } n\toi.
\end{equation*}

We denote by $\mu_G$ the measure $\mu_G (x, \infty) = -\log G(x)$, $x\in \EE$.
Consider  point processes
\begin{equation*}%\label{def_N_t_general}
	N_n = \sum_{i\in \NN} \delta_{\left(\frac{i}{n}, \frac{X_i - b_n}{a_n} \right)}\,,
	\quad n \in \NN\,.
\end{equation*}

It is well known, see again \cite{res87}, that
$X\in\MDA(G)$ is both necessary and sufficient for weak convergence of 
$N_n$ towards a limiting point process, $N$  say, which is a $\mathrm{PRM}(\lambda\times \mu_G)$ in
 $M_p([0,\infty)\times \mathbb{E})$, where both $\mathbb{E}$ and the concept of {\em boundedness} depend on $G$. For instance, in the Gumbel $\MDA$,  $\mathbb{E}=(-\infty,\infty)$ and the sets are considered {\em bounded} in $[0,\infty) \times \mathbb{E}$ if contained in some set of the type 
$[0,T] \times (a,\infty)$, $a \in \RR, T>0 $, cf. \cite{bapl}.

Denote by  $m \restr{A}$ the restriction of a point measure $m$ on a
set $A$, i.e. $m \restr{A}(B) = m(A \cap B).$ Denote by $\EE'$ an arbitrary measurable subset of $\R^d$. The following simple lemma, see Lemma 1 in \cite{drago}, plays an important role in 
a couple of our proofs.
\begin{lemm}\label{drago1}
	Assume that $N$, $(N_t)_{t \geq 0}$ are point processes with values in $M_p([0, \infty)\times \mathbb{E}')$.  
	Assume further that $Z$, $(Z_t)_{t\geq 0}$ are $\R_+$ valued random variables. If $P(N(\{Z\} \times \mathbb{E}') > 0) = 0$ and $(N_t, Z_t) \cid (N, Z),$ in the product topology as $t \toi,$ then 
	\begin{align*}
		N_t \restr{[0,Z_t] \times \mathbb{E}'} \cid N \restr{[0,Z] \times \mathbb{E}'} \quad \text{ as } t\to \infty.
	\end{align*}
\end{lemm}

Suppose that the stopping time $K$ is almost surely finite. Our analysis of $H$ depends on the following simple observation: since $((W_j, X_{j}))_{j\in \NN}$ is an i.i.d. sequence, by the strong Markov property, after stopping time $K_1 = K$, the sequence $((W_{K_1+j},$ $X_{K_1+j}))_{j\in \NN}$ has the same distribution as the original sequence. Therefore it has its own stopping time $K_2$ distributed as $K_1$ such that $((W_{K_1+K_2+j}, X_{K_1+K_2+j}))_{j\in \NN}$ again has the same distribution. Using the shift operator $\vartheta$, one can also write
$K_2 = K \circ \vartheta^{K_1}(((W_j, X_{j}))_{j})$.
 Applying this argument iteratively, we can break the original sequence into i.i.d. blocks
\begin{gather*}
	((W_{T(l-1)+1}, X_{T(l-1)+1}),\, (W_{T(l-1)+2}, X_{T(l-1)+2}),\, \dots,\, (W_{T(l)}, X_{T(l)}))_{l\in \NN}, \nonumber\\
\mbox{where} \quad T(0)=0, \quad T(n) = K_1+K_2+\dots+K_n.
\end{gather*}
Clearly,
$$ H_l = \bigvee_{j=T(l-1)+1}^{T(l)} X_j\,,\qquad  l \in \NN,$$
are i.i.d. with the same distribution as the original compound maximum $H$.
Assume that $((W_{i,j}, X_{i,j}))_{i,j\in \NN}$ is an i.i.d. array of elements as above and let $(K_i')_{i\in \NN}$ be an i.i.d. sequence of stopping times such that for each $l\in \NN$, $(K_l', (W_{l,j}, X_{l,j})_{j\in \NN})\eind (K, (W_j, X_j)_{j\in \NN})$. Then
$$ H'_l = \bigvee_{j=1}^{K'_l} X_{l,j}$$
are also i.i.d. with the same distribution as $H$.
Before stating the main theorem, we prove a simple lemma.

\begin{lemm}\label{lemma:point:proc}
Assume that $\xi = \EE [K]< \infty$. Then
\begin{align*}
    \sum_{i=1}^n \sum_{j=1}^{K'_i} \delta_{\frac{X_{i,j}-b_{\lfloor n\xi \rfloor}}{a_{\lfloor n\xi \rfloor}}} \cid \PRM(\mu_G) \quad \text{ as } n\to \infty.
\end{align*}
\end{lemm}
\begin{proof}
    First note that
    $$ \sum_{i=1}^n \sum_{j=1}^{K'_i} \delta_{\frac{X_{i,j}-b_{\lfloor n\xi \rfloor}}{a_{\lfloor n\xi \rfloor}}} \eind \sum_{i=1}^{T(n)} \delta_{\frac{X_i-b_{\lfloor n\xi \rfloor}}{a_{\lfloor n\xi \rfloor}}}.$$
    To use Lemma \ref{drago1}, let $Z=1$, $(Z_n)_{n\in \NN} = (T(n)/(n\xi))_{n\in \NN}$ be an $\R_+$-valued  random variables,  $N=\PRM(\lambda\times \mu_G)$ as before and define point processes $(N_n')_{n\in \NN}$, where
    $$ N_n' = \sum_{i\in \NN} \delta_{\left(\frac{i}{n\xi}, \frac{X_{i}-b_{\lfloor n\xi \rfloor}}{a_{\lfloor n \xi \rfloor}}\right)}$$
    with values in the space %of Radon point measures 
    $[0, \infty)\times \mathbb{E}$, where the space $\mathbb{E}$ depends on $G$ as before. By the weak law of large numbers and by Proposition 3.21 from \cite{res87}, since $X_1\in \MDA(G)$, we have
    $$ Z_n \cip Z=1\quad \text{and} \quad N_n'\cid N \quad \text{ as } n\to \infty.$$
    Hence, by the standard Slutsky argument (Theorem~3.9 in \cite{billingsley})
    $$ (N_n',\, Z_n) \cid (N,\,Z) \quad \text{ as } n\to \infty.$$
    Note that $ \PP\left(N(\{Z\}\times \mathbb{E} ) >0\right) = 0$, so by Lemma \ref{drago1}, 
	\begin{align*}
	N_n'\Big|_{[0, Z_n]\times \mathbb{E}} \cid N\Big|_{[0, Z]\times \mathbb{E}}. 
	\end{align*}
	Conclude that 
	\begin{align*}
	N_n'\Big|_{\left[0, \frac{T(n)}{n\xi}\right]\times\mathbb{E}} ([0,\infty)\times \cdot\,) &= \sum_{i=1}^{T(n)} \delta_{\frac{X_i-b_{\lfloor n\xi \rfloor}}{a_{\lfloor n\xi \rfloor}}}(\, \cdot \,)  \cid N\Big|_{[0,1]\times \mathbb{E}}([0,\infty)\times \cdot\,)  \quad \text{ as } n\to \infty,
	\end{align*}
	where the point process on the right is a $\PRM(\mu_G)$, see Theorem 2 in \cite{drago} for details.
\end{proof}

\begin{theo}\label{lem:max2}
	Assume that  $K$ is a stopping time with respect to the filtration $(\mathcal{F}_j)_{j\in \NN}$ with a finite mean. 
 
	If $X$  belongs to  $\MDA(G)$, then the same holds for $H=\bigvee_ {j=1}^K X_j$.
\end{theo}
\begin{proof}

    For $(H_i)$ i.i.d. copies of $H$, using Lemma \ref{lemma:point:proc} and the notation therein,
\begin{align*}
    \PP\left(\frac{\bigvee_{i=1}^n H_i - b_{\lfloor n \xi\rfloor}}{a_{\lfloor n\xi \rfloor}}\leq x\right) &= \PP \left(  \sum_{i=1}^n \sum_{j=1}^{K'_i} \delta_{\frac{X_{i,j}-b_{\lfloor n\xi \rfloor}}{a_{\lfloor n\xi \rfloor}}} (x, \infty) = 0 \right) \\
    &\to \PP\left(\PRM(\mu_G)(x,\infty)=0\right) = G(x).
\end{align*}
\end{proof}

\begin{customexample}{2.1 (continued)}
Provided $\EE [W] < \infty$, we recover known results for example (a). Since $\EE [K]< \infty$, in the case (b)  $H$ belongs to the same $\MDA$ as $X$. As we have seen, the case (c) is more involved, but the theorem implies that  if $W_1$ has a heavier tail index than $X$, then $\EE [K] =\infty$ and $H \not \in \MDA(G)$. On the other hand, for bounded or lighter tailed $W$,  we can still have $H  \in \MDA(G)$. 
\end{customexample}

\section{Limiting behaviour of the maximal claim size in the marked renewal cluster model}\label{sec:GMEM}

To describe the marked renewal cluster model consider first an independently marked renewal process $N^0$. Let $(Y_k)_{ k\in \NN }$ be a sequence of i.i.d. non-negative inter-arrival times in $N^0$ and $(A_k)_{ k\in \NN }$ be i.i.d. marks independent of $(Y_k)_{ k\in \NN }$ with distribution $Q$ on $(\mathbb{S}, \BB(\mathbb{S}))$. Throughout we assume that 
\begin{equation*}%\label{nu}
0<\EE [Y] = \frac{1}{\nu} < \infty.
\end{equation*} 
If we denote by $(\Gamma_i)_{i\in \NN}$ the sequence of partial sums of  $(Y_k)_{ k\in \NN}$, the process $N^0$ on the space $[0,\infty) \times \mathbb{S}$\, has the representation
$$
N^0 = \dsum_{i\in \NN} \delta_{ \Gamma_i,A_i}\,.
$$
The processes of this type appear in a non--life insurance mathematics where marks are often referred to as claims. They can represent the size of the claim, type of the claim, severity of the accident, etc.

%Denote the space of boundedly finite point measures on this space by $M_p = M_p ( \left[ 0,\infty \right)$ $\times \, \mathbb{S})$ and 
Assume that at each time $\Gamma_i$ with mark $A_i$ another
point process in $M_p([0,\infty)\times \mathbb{S}) $, denoted by $G_i$, is generated. All $G_i$'s are mutually independent and intuitively represent clusters of points {superimposed} on $N^0$ after time $ \taui$. Formally, there exists a probability kernel $K$ from $\mathbb{S}$ to $M_p([0,\infty)\times \mathbb{S})$ such that, conditionally on $N^0$, the point processes $G_i$ are independent, a.s. finite and with the distribution equal to  $K(A_i,\cdot)$. Note that this permits dependence between  $G_i$ and $A_i$. 

In this setting, the process $N^0$ is usually called the parent process, while $G_i$ are called the descendant processes. We can write
$$
G_i = \dsum_{j= 1}^{K_i}  \delta_{T_{i,j},A_{i,j}}\,,
$$ 
where ${(T_{i,j})}_{j\in \NN}$ is a sequence of non-negative random variables and $K_i$ is an $\ZZ_+ $--valued random variable. If we count the original point arriving at time $\taui$, the actual cluster size is
$K_i+1$.

Throughout, we also assume that the cluster processes $G_i$ are independently marked with the same mark distribution $Q$ independent of $A_i$, so that all the marks $A_{i,j}$ are i.i.d. Note that $K_i$ may possibly depend on $A_i$. We assume throughout that 
\begin{equation*} \label{Kfin}
\Exp [K_i] <\infty\,.
\end{equation*}
Finally,  to describe 
the size and other characteristics of all the observations (claims) together with their  arrival times, we
use a marked point process   $N$ as a random element in $M_p([0,\infty)\times\mathbb{S})$ of the
form
\begin{equation}\label{e:PoisProc}
N= \dsum_{i=1}^\infty  \dsum_{j = 0}^{K_i} \delta_{ \taui+T_{i,j},A_{i,j}}\,,
\end{equation}
where we set $T_{i,0} = 0$ and $A_{i,0} = A_i$. In this representation, the claims arriving at time $ \taui$ and corresponding to the index $j=0$ are called ancestral or  immigrant claims, while the claims arriving at times $ \taui+T_{i,j},\ j \in \NN$, are referred to as progeny or offspring. 
Note that $N$ is a.s. boundedly finite because $\Gamma_i \toi$ as $i \toi$ and $K_i$ is a.s. finite for every $i$, so one could also write 
\begin{equation}\label{nocluster}
N = \dsum_{k=1}^\infty \delta_{\tau_k,A^k} \,,
\end{equation}
with  $\tau_k\leq \tau_{k+1}$ for all $k \in \NN$ and $A^k$ being i.i.d. marks which are in general not independent of the arrival times $(\tau_k)$. {Observe that this representation ignores the information regarding the clusters of the point process.} Note also that eventual ties turn out to be irrelevant  asymptotically.

In the special case, when the inter-arrival times are exponential with parameter $\nu,$ the renewal counting process which generates the arrival times in the parent process is a homogeneous Poisson process. Associated marked renewal cluster model is then called marked Poisson cluster process, see \cite{daley}, cf. \cite{mi}.

\begin{rema}
	In all our considerations we take into account the original immigrant claims arriving at times $\taui$ as well. One could of course ignore these claims and treat  $\taui$  as times of incidents that trigger, with a possible delay, a cluster of subsequent payments as in the model of the so called incurred but not reported (IBNR) claims, cf. \cite{mikosch}.

\end{rema}

The numerical observations, i.e. the size of the claims,
are produced by an application of a measurable function on the marks, say 
$f:\mathbb{S} \to \mathbb{R}_+$. 
The maximum of all claims due to the arrival of an immigrant claim at time $ \taui$
equals
\begin{equation} \label{e:Ei}
H_i = \bigvee_{j=0}^{K_i} X_{i,j}, %\max  \left\lbrace X_{i,j}: 0\leq j \leq K_i\right\rbrace  \,,
\end{equation}
where $X_{i,j} = f(A_{i,j})$ are i.i.d. random variables for all $i$ and $j$. 
The random variable $H_i$ has an interpretation as the maximal claim size coming from the $i$th immigrant and its progeny. 
If we denote $f(A^k)$ by $X^k$, the maximal claim size in
the period $[0,t]$ can be represented as 
$$
M(t) = \sup \left\lbrace X^k: \tau_{k} \leq t\right\rbrace \,.
$$
In order to bring the model in the context of Theorem \ref{lem:max2}, observe that one can let $W_k = A^k$, for $k\in \NN$.
Introduce the first passage time process $(\tau(t))_{ t\geq 0 }$ defined by
\begin{equation*}%\label{eq:tauB}
\tau(t)  = \inf \left\{ n : \taun > t \right\},\ t \geq 0\,.
\end{equation*}

This means that $\tau(t)$ is the renewal counting process generated by the sequence $(Y_n)_{ n\in \NN}.$ 
According to the strong law for counting processes (Theorem 5.1. in \cite{gut}, Chapter 2), for every $c \geq 0,$
\begin{align*}
\frac{\tau (tc)}{\nu t} \cas c \quad \text{ as } t\to \infty.
\end{align*}
Denote by 
$$M^{\tau}(t) = \bigvee_{i=1}^{\tau(t)} H_i, $$%\sup \{ H_i: 1 \leq i \leq \tau(t)\},$$ 
the maximal claim size coming from the maximal claim sizes in the first $\tau(t)$ clusters.
Now we can write 
\begin{equation} \label{e:Moft} 
M^{\tau}(t) = M(t) \bigvee H_{\tau(t)} \bigvee  \vep_t\,,\quad t \geq 0\,,
\end{equation} 
where the last error term represents the leftover effect at time $t$, i.e.  the maximum of all claims arriving after $t$ which correspond to the progeny of immigrants arriving before time $t$, more precisely
\begin{align*}
\vep_t =  \max \{X_{i,j}: 0\leq \taui \leq t , \, t < \taui+T_{i,j}\},    \quad t \geq 0\,.
\end{align*}
Denote the number of members in the set above by 
\begin{equation}\label{eq:jt}
J_t=\# \{(i,j):0\le \taui \leq t , \, t < \taui + T_{i,j}\}.
\end{equation}

We study the limiting behaviour of the maximal claim size $M(t)$ up to time $t$ and aim to  find sufficient conditions under which $M(t)$ converges in distribution to a non-trivial limit after appropriate centering and normalization.

Recall, $H$ belongs to $\MDA(G)$ if there exist constants $c_n>0,$ $d_n \in \RR$ such that for each $x\in \EE=\{y\in \R: G(y)>0\}$,

\begin{equation}\label{MDA2}
n \cdot \PP \left( H > c_{n}x + d_{n} \right)  \rightarrow - \log G(x) \quad \text{ as } n\to \infty.
\end{equation}
An application of Lemma \ref{drago1} yields the following result.

\begin{propo}\label{prop:21}
	Assume that $ H$ belongs to  $\MDA(G)$ so that \eqref{MDA2} holds and that the error term in \eqref{eq:jt} satisfies 
	\begin{equation*}%\label{condM}
	J_t
	= o_P(t).
	\end{equation*}
	Then
	\begin{equation}\label{MDA}
	\frac{M(t) 
		- d_{\lfloor \nu t \rfloor}
	}{ 
		c_{\lfloor \nu t \rfloor} }
	\cid G \quad \text{ as } t\to \infty.
	\end{equation}
	
\end{propo}

\begin{proof}
    Using the equation \eqref{e:Moft},
	$$\frac{M^{\tau}(t) - d_{\lfloor \nu t \rfloor}}{c_{\lfloor \nu t \rfloor}} 
	= \frac{M(t) - d_{\lfloor \nu t \rfloor}}{c_{\lfloor \nu t \rfloor}} \bigvee \frac{H_{\tau(t)} - d_{\lfloor \nu t \rfloor}}{c_{\lfloor \nu t \rfloor}} \bigvee \frac{\vep_t - d_{\lfloor \nu t \rfloor}}{c_{\lfloor \nu t \rfloor}}.$$
	Since for $x\in \EE$, 
	
	\begin{align*}
	0 &\leq \PP \left( \frac{M^{\tau}(t) - d_{\lfloor \nu t \rfloor}}{c_{\lfloor \nu t \rfloor}} > x \right)  - \PP \left( \frac{M(t) - d_{\lfloor \nu t \rfloor}}{c_{\lfloor \nu t \rfloor}} > x \right)\\ &\leq \PP \left( \frac{H_{\tau(t)} - d_{\lfloor \nu t \rfloor}}{c_{\lfloor \nu t \rfloor}} > x \right) + \PP \left( \frac{\vep_t - d_{\lfloor \nu t \rfloor}}{c_{\lfloor \nu t \rfloor}} > x\right)
	\end{align*}
	it suffices to show that 
	\begin{equation}\label{eq:RVmax}
	\frac{M^{\tau}(t) - d_{\lfloor \nu t \rfloor}}{c_{\lfloor \nu t \rfloor}} \cid G \quad \text{ as } t\to \infty,
	\end{equation}
	\begin{equation}\label{hi}
	\lim_{t \toi} \PP \left( \frac{H_{\tau(t)} - d_{\lfloor \nu t \rfloor}}{c_{\lfloor \nu t \rfloor}} > x \right) = 0, \quad \text{and} \quad \lim_{t \toi} \PP \left( \frac{\vep_t - d_{\lfloor \nu t \rfloor}}{c_{\lfloor \nu t \rfloor}} > x\right) = 0.
	\end{equation}

	Recall that $H_i$ represents the maximum of all claims due to the arrival of an immigrant claim at time $ \taui$ and by \eqref{e:Ei}, it equals
	$$
	H_i =\bigvee_{i=0}^{K_i} X_{i,j}  \,.
	$$
	Note that $(H_i)$ is an i.i.d. sequence because the ancestral mark in every cluster comes from an independently marked renewal point process.

	As in the proofs of Lemma \ref{lemma:point:proc} and Theorem \ref{lem:max2},
	\begin{align*}
	\PP\left(\frac{M^{\tau}(t) - d_{\lfloor \nu t\rfloor}}{c_{\lfloor \nu t \rfloor}}\leq x\right) &= \PP \left(  \sum_{i=1}^{\tau(t)} \delta_{\frac{H_{i}-d_{\lfloor \nu t\rfloor}}{c_{\lfloor \nu t \rfloor}}} (x, \infty) = 0 \right) \\
    &\to \PP\left(\PRM(\mu_G)(x,\infty)=0\right) = G(x),
    \end{align*}
    as $t\toi$, which shows \eqref{eq:RVmax}.

	To show \eqref{hi}, note that $\left\lbrace \tau(t) = k \right\rbrace \in \sigma(Y_1, \dots Y_k)$ and by assumption $\left\lbrace H_{k} \in A \right\rbrace$ is independent of $\sigma(Y_1, \dots Y_k)$ for every $k$. Therefore, $ H_{\tau(t)} \eind H_1\in \MDA(G)$ so the first part of \eqref{hi} easily follows from \eqref{MDA2}. For the second part of \eqref{hi}, observe that the leftover effect $\vep _t$ % = \max \{X_{i,j}: 0\leq \taui \leq t ,  t < \taui+T_{i,j}\}$ 
	admits the following representation 
	$$\vep_t \eind \bigvee_{i=1}^{J_t}X_i,$$	
	for $(X_i)_{i\in \NN}$ i.i.d. copies of $X=f(A)$. 
	Hence, $$\frac{\vep _t - d_{\lfloor \nu t \rfloor}}{c_{\lfloor \nu t \rfloor}} \eind \frac{\bigvee_{i=1}^{J_t}X_i - d_{\lfloor \nu t \rfloor}}{c_{\lfloor \nu t \rfloor}}.$$

	Since $J_t = o_P(t)$, for every fixed $\delta>0$ and $t$ large enough, $\PP(J_t>\delta t)< \delta.$
	For measurable $A=\{J_t>\delta t\}$ we have
	\begin{align*}
	    \PP\left(\dfrac{\bigvee_{i=1}^{J_t}X_i - d_{\lfloor \nu t \rfloor}}{c_{\lfloor \nu t \rfloor}}  > x\right) &\leq \PP(A) + \PP \left( \left\lbrace  \dfrac{\bigvee_{i=1}^{J_t}X_i - d_{\lfloor \nu t \rfloor}}{c_{\lfloor \nu t \rfloor}}  > x \right\rbrace \cap A^C \right)\\
	    &< \delta + \PP\left(\dfrac{\bigvee_{i=1}^{\delta t}X_i - d_{\lfloor \nu t \rfloor}}{c_{\lfloor \nu t \rfloor}}  > x\right)
	\end{align*}
	which converges to $0$, as $\delta\to 0$. 
\end{proof}

As we have seen above, it is relatively easy to determine asymptotic behaviour of the maximal claim size $M(t)$ as long as one can determine tail properties of the random variables $H_i$ and the number of points in the leftover effect at time $t,$ $J_t$ in \eqref{eq:jt}.
An application of  Theorem~\ref{lem:max2}, immediately yields the following corollary.
\begin{coro}\label{cor:41}
	Let $J_t = o_P(t)$ and let $(X_{i,j})$ satisfy \eqref{mda(g)} and the assumptions from the proof of Theorem \ref{lem:max2}. Then \eqref{MDA} holds with $(c_n)$ and $(d_n)$ defined by
	\begin{equation}\label{norm:sequences}
	    (c_n) = (a_{\lfloor (\EE [K] +1) \cdot n \rfloor}), \quad (d_n) = (b_{\lfloor (\EE [K]+1) \cdot n \rfloor}) .
	\end{equation}

\end{coro}

As we shall see in the following section, showing that $J_t = o_P(t)$ holds  remains a rather technical task. However, this can be done for several frequently used cluster models.

\section{Maximal claim size for three special models}\label{sec:APP}

 In this section we present three special models belonging to the general marked renewal cluster model introduced in Section 3. We try to find sufficient conditions for those models in order to apply Proposition \ref{prop:21}.

\begin{rema}\label{stac}
	In any of the three examples below, the point process $N$
	can be made stationary if we start the construction in \eqref{e:PoisProc}  on the state space
	$\RR \times \mathbb{S}$ with a renewal process $ \sum_i \delta_{\taui}$ on the whole real line. For the resulting stationary cluster process we use the notation $N^*$. Still, from the applied perspective, it seems more interesting 
	to study the nonstationary version where both the parent process $N^0$ and the cluster process itself have arrivals only from some point onwards, e.g. in the interval $[0,\infty)$.
	
\end{rema}

\subsection{Mixed binomial cluster model}
Assume that the individual clusters  have the following form
$$
G_i = \dsum_{j=1}^{K_i} \delta_{V_{i,j}, A_{i,j}}\,.
$$

Assume that $(K_i,(V_{i,j})_{j\in \NN},(A_{i,j})_{j\in \ZZ_+})_{i\in \NN}$ constitutes an i.i.d. sequence with the following properties for fixed $i\in \NN$
\begin{itemize} 
	\item $(A_{i,j})_{j\in \ZZ_+}$ are i.i.d.,
	\item $(V_{i,j})_{j\in \NN}$ are conditionally i.i.d. given $A_{i,0}$,
	\item $(A_{i,j})_{j\in \NN}$ are independent of $(V_{i,j})_{j\in \NN}$,
	\item $K_i$ is a stopping time with respect to the filtration generated by the $(A_{i,j})_{j\in \ZZ_+},$ i.e. for every $k \in \ZZ_+$,
	$\{ K_i = k\} \in \sigma (A_{i,0}, \dots A_{i,k}).$
	
\end{itemize}

Notice that we do allow possible dependence between $K_i$ and ${(A_{i,j})}_{j\in \ZZ_+}$. Also, we do not exclude the possibility of dependence between $(V_{i,j})_{j\in \NN}$ and the ancestral mark $A_{i,0}$ (and consequently $K_i$).
Recall, $K$ is an integer valued random variable representing the size of a cluster such that $\EE[K]<\infty$. Observe that we use notation $V_{i,j}$ instead of $T_{i,j}$ to emphasize a relatively simple structure of clusters in this model in contrast with the  other two models in this section. Such a process $N$ is a marked version of the so--called Neyman--Scott  process, e.g. see {Example 6.3 (a) in \cite{daley}}.

\begin{coro} \label{cor:maxmix}
Assume that $f(A)=X$ belongs to $\MDA(G)$ so that \eqref{mda(g)} holds. Then \eqref{MDA} holds for $(c_n)$ and $(d_n)$ defined in \eqref{norm:sequences}.
\end{coro}
\begin{proof}
	Using 
	Theorem \ref{lem:max2} we conclude that the maximum $H$ of all claims in a cluster belongs to the $\MDA$ of the same distribution as $X$. 
	
	Apply Proposition \ref{prop:21} after observing that $J_t=o_P(t)$.

	Using the Markov's inequality, it is enough to check that $\EE[J_t]=o(t)$,

	\begin{align*}
	    \EE [J_t]&=\EE [\# \{(i,j):0\le \taui \leq t ,\, t< \taui + V_{i,j}\}] \\
	    &= \EE\left[\sum_{0\le \taui \leq t }\sum_{j=1}^{K_i}\1_{t\le  \taui +V_{i,j}}\right].
	\end{align*}
	Using Lemma 7.2.12. in \cite{mikosch} and calculation similar as in the proofs of Corollaries 5.1. and 5.3. in \cite{mi}, we have

	\begin{align*}
	    \EE\left[\sum_{0\le \Gamma_i\le t}\sum_{j=1}^{K_i}\1_{t< \Gamma_i+V_{i,j}}\right] &= \int_0^t\EE\left[\sum_{j=1}^{K_i}\1_{V_{i,j}> t-s} \right]\nu ds = \int_0^t\EE\left[\sum_{j=1}^{K_i}\1_{V_{i,j}> x} \right]\nu dx.
	    %&\leq \int_0^t\EE\left[K_i \1_{\text{max}\{V_{i,j}: 1 \leq j \leq K_i\} > x } \right]\nu dx
	\end{align*}

	Now note that as $x \toi$, by the dominated convergence theorem,
	$$\EE\left[\sum_{j=1}^{K_i}\1_{V_{i,j}> x} \right] \to 0 \,.$$

	An application of Ces\`aro argument yields now that
	$\EE[J_t]/t \to 0$.
\end{proof}

\subsection{Renewal cluster model} 

Assume next that the clusters  
$G_i$ 
have the following  distribution
$$
G_i = \dsum_{j=1}^{K_i} \delta_{T_{i,j}, A_{i,j}}\,,
$$
where 
$(T_{i,j})$ 
represents the sequence such that
$$
T_{i,j} = V_{i,1} + \cdots +  V_{i,j},\quad 1\leq j\leq K_i\,.
$$
We keep all the other assumptions from the model in the previous subsection. 

A general unmarked model of the similar type is called Bartlett–-Lewis
model and is analysed in \cite{daley}, see Example 6.3 (b).
See also \cite{fayetal} for an application of a similar point process to modelling of teletraffic data.
By adapting the arguments from the Corollary \ref{cor:maxmix} we can easily obtain the next Corollary.
\begin{coro} \label{cor:maxren}
	Assume that $f(A)=$ $X$ belongs to $\MDA(G)$ so that \eqref{mda(g)} holds. Then \eqref{MDA} holds for $(c_n)$ and $(d_n)$ defined in \eqref{norm:sequences}.
\end{coro}

\subsection{Marked Hawkes processes}

One of the example in our analysis is the so called (linear) marked Hawkes process. 
They are typically introduced through their stochastic intensity (see, for example \cite{zhu} or \cite{daley}). More precisely, a  point process $N = \sum_{k} \delta_{\tau_k,A^k} \,,$ represents a Hawkes process of this type if the random marks $(A^k)$ are i.i.d. with distribution $Q$ on the space $\mathbb{S}$, while the arrivals $(\tau_k)$ have the stochastic intensity of the form 
$$\lambda(t) = \nu + \sum_{\tau_i< t} h(t-\tau_i,A^i)\,,$$
where $\nu > 0$ is a constant and  $h:[0,\infty)\times \mathbb{S}\to \mathbb{R}_+$ is assumed to be integrable  in the sense that
$\int_0^\infty \EE [h(s,A)] ds < \infty$.

On the other hand, Hawkes processes of this type have a neat Poisson cluster representation due to \cite{HaOa}. For this model,
the clusters $G_{i}$ are recursive aggregation of Cox processes, i.e.
Poisson processes with random mean measure $ \tilde{\mu}_{A_i}  \times Q$ where
$  \tilde{\mu}_{A_i} $ has the following form
\begin{align*}
\tilde{\mu}_{A_i} (B) = \dint_B h(s, A_i) ds\,,
\end{align*}
for some  fertility (or self--exciting) function $h$, cf. {Example 6.4 (c) of \cite{daley}}. 
It is useful to
introduce a time shift operator $\theta_t$, by denoting  
$$
\theta _t  m =  \sum_j \delta_{t_j+t,a_j}\,,
$$ for an arbitrary point measure $m = \sum_j \delta_{t_j,a_j} \in M_p([0,\infty)\times \mathbb{S})$ and $t\geq 0$.
Now, for the parent process $N^0 = \sum_{i\in \NN} \delta_{ \taui,A_i}\,$ which is a Poisson point process with mean measure $\nu \times Q$ on the space
$[0,\infty) \times \mathbb{S}$\,,
the cluster process corresponding to a point $( \Gamma_i,A_i)$
satisfies the following recursive relation
\begin{equation} \label{GA}
G_i =  \dsum_{l=1}^{{L_{A_i}}}  \left(  {\delta_{\tau^1_l,A^1_l} +} \theta_{\tau^1_l} G^1_l \right)\,,
\end{equation}
where, given $A_i,$ ${\tilde{N}}_i = \sum_{l=1}^{L_{A_i}} \delta_{\tau^1_l,A^1_l}$ is a Poisson process with mean measure 
$ \tilde{\mu}_{A_i}  \times Q,$
the sequence $(G^1_l)_l$ is i.i.d., distributed as $G_i$ and independent of ${\tilde{N}}_i$. 

Thus, at any ancestral point $(\Gamma_i,A_i)$  a cluster of points appears as a whole cascade of points to the right in time generated recursively according to \eqref{GA}.  Note that $L_{A_i}$ has Poisson distribution conditionally on $A_i$, with mean $\kappa_{A_i}=\int_0^\infty h(s,A_i) ds$. It corresponds to the number of the first generation progeny $(A^1_l)$ in the cascade. Note also that the point processes forming the second generation are again Poisson conditionally on the corresponding first generation  mark $A_l^1$. The cascade $G_i$ corresponds to the process formed by the successive generations, drawn recursively as Poisson processes given the former generation.
The marked Hawkes process is obtained by attaching to the ancestors $(\taui,A_{i})$ of the marked Poisson process $
N^0 = \sum_{i\in \NN} \delta_{ \taui,A_i}
$ a cluster of points, denoted by $C_i$, which contains point $(0,A_{i})$  and a whole cascade $G_i$ of points to the right in time generated recursively according to \eqref{GA} given $A_i$.
Under the assumption
\begin{equation} \label{e:kappa}
\kappa =  \Exp \left[\dint h(s,A) ds\right] < 1\,,
\end{equation}
the total number of points in a cluster is generated by a subcritical branching process. Therefore, 
the clusters are finite  almost surely. Denote their size by
$K_i {+1}$. It is known (see Example 6.3.(c) in \cite{daley}) that under \eqref{e:kappa} the clusters always satisfy
\begin{equation}\label{e:expecK}
\Exp [K_i] {+1} = \frac{1}{1-\kappa}\,.
\end{equation}
Note that the clusters $C_i$, i.e. point processes which represent a cluster together with the mark $A_i$ are independent by construction. They can be represented as
\begin{align*}
C_i= \dsum_{j = 0}^{K_i} \delta_{\taui+T_{i,j},A_{i,j}}\,,
\end{align*}
with $A_{i,j}$ being i.i.d., $A_{i,0}=A_i$ $T_{i,0} = 0$ and $T_{i,j},$  $j\in \NN$, representing arrival times of progeny claims in the cluster $C_i$. Observe that in the case when marks do not influence conditional density, i.e. when $h(s,a) = h(s)$, the random variable
$K_i{+1}$ has a so-called Borel distribution with parameter $\kappa$, see \cite{borel}.
Notice also that  in general, marks and arrival times of the final Hawkes process $N$  are not independent  of each other, rather, in the terminology of \cite{daley}, the marks in the process $N$ are only unpredictable.

As before, the maximal claim size in one cluster is of the form
$$H \stackrel{d}{=} \bigvee_{j=0}^K X_j .$$

Note that $K$ and $(X_j)$ are not independent. 
In this case, due to the representation of Hawkes processes as the recursive aggregation of Cox processes \eqref{GA}, maximal claim size can also be written as 
$$H \stackrel{d}{=} X \vee \bigvee_{j=1}^{L_A} H_j\,.$$

Recall from (\ref{e:kappa}) that $\kappa = \Exp [\kappa_A] < 1.$ The $H_j$'s on the right hand side are independent of $\kappa_A$ and i.i.d. with  the same distribution as $H$. Conditionally on $A$ the waiting times are i.i.d.  with common density 
	\begin{equation}\label{densi}
	\frac{h(t,A)}{\kappa_A}\,, \quad t\ge0\,,
	\end{equation}
	see \cite{zhu} or \cite{mi}. 
In order to apply Proposition \ref{prop:21} first we show that $H$ is in the $\MDA(G)$ using the well known connection between branching processes and random walks, see for instance \cite{asmussen_foss}, \cite{ben00} 
or quite recently \cite{costa20}. This is the subject of the next Lemma. 

\begin{lemm}\label{gwbt}
	Let $X$ belongs to $\MDA(G)$ in the marked Hawkes model. Then $H$ also belongs to the same $\MDA(G).$
\end{lemm}

\begin{proof}	
Due to the recursive relation \eqref{GA}, each cluster can be associated with a subcritical branching process (Bienaym\'e--Galton--Watson tree) where 

the total number of points in a cascade (cluster) corresponds to the total number of vertices in such a 
tree. 
It has the same distribution as the first hitting time of level $0$
$$
\zeta = \inf \left\lbrace k: S_k = 0 \right\rbrace, 
$$ 
by a random walk $(S_n)$ defined as
\begin{align*}
S_0 = 1,\quad S_n = S_{n-1} + L_n - 1\,,
\end{align*}
with i.i.d. $L_n \stackrel{d}{=} L.$ Notice, $(S_n)$ has negative drift which leads to conclusion that $\zeta$ is a proper random variable. 

Moreover, since $\EE [L] < 1$ an application of Theorem 3 from \cite{gut} gives $\EE [\zeta] < \infty$ and that we can use \eqref{e:expecK} since $\zeta = K+1.$

If we write, for arbitrary $k\in \NN$,
\begin{align*}
\left\lbrace \zeta = k \right\rbrace &= \left\lbrace S_0 >0, S_1 >0, \dots, S_{k-1} >0, S_k=0 \right\rbrace \\
&=\left\lbrace 1>0, L_1>0, \dots, \dsum_{i=1}^{k-1} L_i - (k-2) >0, \dsum_{i=1}^{k} L_i - (k-1) =0 \right\rbrace \\
&\in \sigma \left( L,A_0,A_1,\dots, A_k \right),      
\end{align*}
we see that $\zeta$ is a stopping time with respect to $(\mathcal{F'}_j)_{j\in \ZZ_+},$ where  $\mathcal{F'}_j = \sigma(L, A_0, A_1, \dots,\\$ $A_j),$ and where $L$ has conditionally Poisson distribution with random parameter $\kappa_A$ and is independent of the sequence $(A_j)_{j\in \ZZ_+}.$ 
By the Theorem \ref{lem:max2} we conclude that $H$ is also in the $\MDA(G)$.
\end{proof}

\begin{rema}
    The equation \eqref{e:expecK} implies that the sequences $(c_n)$ and $(d_n)$ in the following corollary have the representation:
	$(c_n) = (a_{\lfloor \frac{1}{1-\kappa} n \rfloor})$ and $(d_n) = (b_{\lfloor \frac{1}{1-\kappa} n \rfloor})$.
\end{rema} 

\begin{coro} \label{thm:hmax1}
	Assume that $X$ belongs to $\MDA(G)$ so that \eqref{mda(g)} holds and 
	\begin{align*}
	\EE\left[ \tilde{\mu}_A(t,\infty)\right] \to 0 \quad \text{ as } t\to \infty.
	\end{align*}	 
	Then \eqref{MDA} holds for $(c_n)$ and $(d_n)$ defined in \eqref{norm:sequences}.
\end{coro}

\begin{proof}
	Recall from \eqref{nocluster} that one can write
		$$
		N= \dsum_{i=1}^\infty  \dsum_{j = 0}^{K_i} \delta_{ \taui+T_{i,j},A_{i,j}} = \dsum_{k=1}^\infty \delta_{\tau_k,A^{k}} \,,
		$$
		w.l.o.g. assuming that $0\leq \tau_1\leq \tau_2\leq \ldots$.
		At each time $\tau_j$, a claim arrives  generated by one of the previous claims or
		an entirely new (immigrant) claim  appears.
		In the former case, if $\tau_j$ is a direct offspring of a claim at time $\tau_i$, we will write
		$\tau_i\to \tau_j$. Progeny $\tau_j$ then potentially creates further  claims. Notice that 
		$\tau_i\to \tau_j$ is equivalent to $\tau_j=\tau_i+V_{i,k}$, $k\le L^i = L_{A^i}$ where $V_{i,k}$ are waiting times which, according to the discussion above \eqref{densi}, are i.i.d. with common density $h(t,A^i)/ \kappa_{A^i}$, $t\ge0$ and independent of $L^{i}$ conditionally on the mark $A^i$ of the claim at $\tau_i$. 
		Moreover,
		conditionally on $A^i$, the number of direct progeny of the claim at $\tau_i$, denoted  by $L^{i}$, has Poisson distribution with parameter $\tilde{\mu}_{A^i}$.
		We denote by $K_{\tau_j}$ the total number of points generated by the arrival at $\tau_j$. Clearly, $K_{\tau_j}$'s are identically distributed as $K$ and even mutually independent if we consider only points which are not offspring of one another.
	
	It is enough to check $\EE[J_t]/t=o(1),$ and see that

	\begin{align*}
	\EE [J_t] &=  \EE\Big[ \dsum_{\taui \leq t}\dsum_j \1_{\taui + T_{i,j} >t}  \Big]\\
	&= \EE\Big[\dsum_{\tau_i \leq t} \dsum_{\tau_j > t}  (K_{\tau_j}+1) \, \1_{\tau_i\to \tau_j}\Big]\\
	&=\EE\Big[\dsum_{\tau_i \leq t} \EE\Big[\dsum_{k=1}^{L^{i}} (K_{\tau_i+V_{i,k}}+1) \1_{\tau_i+V_{i,k} > t}  \mid (\tau_i,A^i)_{i\ge 0}; \tau_i \le t\Big]\Big]\\
	&=\frac{1}{1-\kappa} \EE  \left[ \dint_{0}^t \dint_{\mathbb{S}} \tilde{\mu}_a((t-s,\infty)) N(ds,da)  \right],
	\end{align*}
	where $ \tilde{\mu}_{a} ((u,\infty)) = \int_u^\infty h(s, a) ds$.
	Observe that  from the
	projection theorem, see \cite{brem}, Chapter 8, Theorem 3, the last expression equals
	\[
	\frac{1}{1-\kappa} \EE\left[\dint_0^t \dint_{\mathbb{S}} \tilde{\mu}_a((t-s,\infty)) Q(da) \lambda(s)ds   \right].
	\]
	
	Recall from the Remark \ref{stac} that $N$ has a stationary version, $N^*$, such that the expression  $\Exp  \left[ \lambda^*(s)   \right]$ is a constant equal to $ \nu/(1-\kappa)$.
	Using Fubini's theorem, one can further bound from above the last expectation by
	\begin{align*}
	\EE\left[\dint_0^t \dint_{\mathbb{S}} \tilde{\mu}_a((t-s,\infty)) Q(da) \lambda^*(s)ds\right]&=\dint_0^t \dint_{\mathbb{S}} \tilde{\mu}_a((t-s,\infty)) Q(da) \EE[\lambda^*(s)]ds\\
	&=\frac{\nu}{1-\kappa}\dint_0^t \dint_{\mathbb{S}} \tilde{\mu}_a((t-s,\infty)) Q(da) ds.
	\end{align*}

	Now, we have 
	\begin{align*}
	\Exp {J_t} \leq \dfrac{\nu  }{(1-\kappa)^{2}}
	\dint_0^t \dint_{\mathbb{S}} \tilde{\mu}_a((t-s,\infty)) Q(da)ds
	= \dfrac{\nu  }{(1-\kappa)^{2}}
	\dint_0^t \int_s^\infty \Exp[h(u,A)]du  ds \,.
	\end{align*}
	
	Dividing the last expression  by $t$ and applying L'H\^opital's rule  proves the theorem for the nonstationary or pure Hawkes
	process.
\end{proof}

\section*{Acknowledgments}
We sincerely thank the anonymous reviewers for suggestions which led to simplified proofs and an improved layout of the article.
Bojan Basrak and Nikolina Milinčević are financed in part by the Croatian-Swiss Research Program of the Croatian Science Foundation and the Swiss National Science Foundation
- grant IZHRZ0 - 180549.

\bibliographystyle{plainnat}

\begin{thebibliography}{9999}

\bibitem[Asmussen and Foss(2018)]{asmussen_foss}
{\sc Asmussen, S., Foss, S.} (2018). Regular variation in a fixed-point problem for single and multi--class branching processes and queues. {\em Advances in Applied Probability} {\bf 50(A)}, 47--61.

\bibitem[Barndorff-Nielsen(1964)]{barndorff} 
{\sc Barndorff-Nielsen, O.} (1964). On the limit distribution of the maximum of a random number of independent random variables. {\em Acta Math. Acad. Sci. Hungar} {\bf 15}, 399--403.

\bibitem[Basrak and Planini\'c(2019)]{baplSPL} 
{\sc Basrak, B., Planini\'c, H.} (2019). A note on vague convergence of measures. {\em Statistics \& Probability Letters} {\bf 153}, 180--186.
%
\bibitem[Basrak and Planini\'c(2021)] {bapl} 
{\sc Basrak, B., Planini\'c, H.} (2021). Compound Poisson approximation for regularly varying fields with application to sequence alignment. {\em Bernoulli} {\bf 27}, 1371--1408.
%
\bibitem[Basrak and \v{S}poljari\'c(2015)]{drago} 
{\sc Basrak, B., \v{S}poljari\'c, D.} (2015). Extremes of random variables observed in renewal times. {\em Statistics \& Probability Letters} {\bf 97}, 216--221.
%
%
\bibitem[Basrak et al.(2019)]{mi} 
{\sc Basrak, B., Wintenberger, O., \v{Z}ugec, P.} (2019). On total claim amount for marked Poisson cluster models. {\em Advances in Applied Probability} {\bf 51.2}, 541--569.
%
\bibitem[Bennies and Kersting(2000)]{ben00} 
{\sc Bennies, J., Kersting, G.} (2000). A Random Walk Approach to Galton–-Watson Trees. {\em Journal of Theoretical Probability} {\bf 13}, 777--803.
%
\bibitem[Berman(1962)]{berman} 
{\sc Berman, S.M.} (1962). Limiting distribution of the maximum term in sequences of dependent random variables. {\em The Annals of Mathematical Statistics} {\bf 33}, 894--908.

\bibitem[Billingsley(1999)]{billingsley}
{\sc Billingsley, P.} (1999). {\em Convergence of Probability Measures}. Second Edition. Wiley. New York. 

\bibitem[Bremaud(1981)]{brem} 
{\sc Bremaud, P.} (1981). {\em Point processes and queues}. Springer Verlag, New York.

\bibitem[Costa et al.(2020)]{costa20} 
{\sc Costa, M., Graham, C., Marsalle, L., Tran, V.} (2020). Renewal in Hawkes processes with self-excitation and inhibition. {\em Advances in Applied Probability} {\bf 52(3)}, 879--915.
%
\bibitem[Daley and Vere Jones(2003)]{daley} 
{\sc Daley, D.J., Vere--Jones, D.} (2003). {\em An introduction to the theory of Point processes} Volume I, II., Second Edition. New York: Springer.
%
\bibitem[de Haan and Ferreira(2006)]{dehaan} 
{\sc de Haan, L., Ferreira, A.} (2006). {\em Extreme value theory. An introduction.} New York: Springer.
%
\bibitem[Embrechts et al.(1997)]{embrechtsetal}
{\sc Embrechts, P., Kl\"uppel\-berg, C., Mikosch, T.} (1997). {\em Modelling Extremal Events.} Springer, Berlin.
%
\bibitem[Fa\"y et al.(2006)]{fayetal} 
{\sc Fa\"y, G., Gonzalez-Arevalo, B., Mikosch, T., Samorodnitsky, G.} (2006). Modeling telegraffic arrivals by a Poisson cluster process. {\em Queueing Systems} {\bf 54}, 121--140.
%
\bibitem[Gut(2009)]{gut} 
{\sc Gut, A.} (2009). {\em Stopped Random Walks}, 2nd Edition. New York: Springer.
%
\bibitem[Haight and Breuer(1960)]{borel} 
{\sc Haight, F.A., Breuer, H.A.} (1960). The Borel-Tanner Distribution. {\em Biometrika} {\bf 47}, 143--150.
%
\bibitem[Hashorva et al.(2020)]{hashorva} 
{\sc Hashorva, E., Padoan, S.A., Rizzelli, S.} (2021). Multivariate extremes over a random number of observations. {\em Scandinavian Journal of Statistics} {\bf 48(3)}, 845--880. 
%
\bibitem[Hawkes and Oakes(1974)]{HaOa} 
{\sc Hawkes, A.G., Oakes, D.} (1974). A cluster process representation of a self - exciting process. {\em Journal of Applied Probability} {\bf 11}, 493--503.
%

\bibitem[Jessen and Mikosch(2006)]{JesMik} 
{\sc Jessen, A.H., Mikosch, T.} (2006). Regularly varying functions. {\em Publications de l'Institut Mathematique} {\bf 80(94)}, 171--192.
%
\bibitem[Kallenberg(2017)]{kallenberg2017} 
{\sc Kallenberg, O.} (2017). {\em Random Measures, Theory and Applications}. Springer International Publishing.
%
\bibitem[Karabash and Zhu(2015)]{zhu} 
{\sc Karabash, D., Zhu, L.} (2015). Limit theorems for marked Hawkes processes with application to a risk model. {\em Stochastic Models} {\bf 31(3)}, 433--451.
%
\bibitem[Meerschaert and Stoev(2009)]{meerschaert} 
{\sc Meerschaert, M.M., Stoev, S.A.} (2009). Extremal limit theorems for observations separated by random power law waiting times. {\em Journal of Statistical Planning and Inference} {\bf 139(7)}, 2175--2188.
%
\bibitem[Mikosch(2009)]{mikosch} 
{\sc Mikosch, T.} (2009). {\em Non life insurance mathematics}, Second Edition. New York: Springer.
%
\bibitem[Pancheva et al.(2009)]{pancheva} 
{\sc Pancheva, E., Mitov, I.K., Mitov, K.V.} (2009). Limit theorems for extremal processes generated by a point process with correlated time and space components. {\em Statistics \& Probability Letters} {\bf 79(3)}, 390--395.
%
\bibitem[Resnick(1987)]{res87} 
{\sc Resnick, S.I.} (1987). {\em Extreme Values, Regular Variation, and Point Processes}. New York: Springer.
%

\bibitem[Silvestrov and Teugels(1998)]{silvestrov} 
{\sc Silvestrov, D.S., Teugels, J.L.} (1998). Limit theorems for extremes with random sample size. {\em Advances in Applied Probability}, {\bf 30(3)}, 777--806.
%
%
\bibitem[Tillier and Wintenberger(2018)]{tillier-wint}
{\sc Tillier, C., Wintenberger, O.} (2018). Regular variation of a random length sequence of random variables and application to risk assessment. {\em Extremes} {\bf 21}, 27--56.
%
\bibitem[Towe et al.(2020)]{towe et al}
{\sc Towe, R., Tawn, J., Eastoe, E., Lamb} (2020). R. Modelling the Clustering of Extreme Events for Short-Term Risk Assessment. {\em JABES} {\bf 25}, 32--53
%
\bibitem[Vogel et al.(2020)]{vogel et al}
{\sc Vogel, M.M., Hauser, M., Seneviratne, S.I.} (2020). Projected changes in hot, dry and wet extreme events' clusters in {CMIP}6 multi-model ensemble. {\em Environmental Research Letters} {\bf 15(9)}, 094021	
%


	
\end{thebibliography}

\end{document}